\documentclass[12pt]{article}
\usepackage{amsmath,amssymb,amsbsy,amsfonts,amsthm,latexsym,amsopn,amstext,amsxtra,euscript,amscd}

\newtheorem{thm}{Theorem}

\newtheorem{lem}[thm]{Lemma}

\def\\{\cr}
\def\({\left(}
\def\){\right)}
\def\[{\left[}
\def\]{\right]}
\def\<{\langle}
\def\>{\rangle}
\def\fl#1{\left\lfloor#1\right\rfloor}

\def\cF{\mathcal F}

\def\cE{\mathcal E}
\def\cG{\mathcal G}
\def\cH{\mathcal H}

\def\Z{\mathbb{Z}}

\def\notdivides{\mathrel{\kern-3pt\not\!\kern3.5pt\bigm|}}

\begin{document}

\title{On the Number of Eisenstein Polynomials
of Bounded Height}

\author{
{\sc Randell Heyman}\\
{Department of Computing, Macquarie University} \\
{Sydney, NSW 2109, Australia}\\
{\tt randell@unsw.edu.au}
\and
{\sc Igor E. Shparlinski}  \\
Department of Computing, Macquarie University \\
Sydney, NSW 2109, Australia\\
{\tt igor.shparlinski@mq.edu.au}
}

\date{ }
\maketitle

\begin{abstract} We obtain a more precise version
of an asymptotic formula of A.~Dubickas for the number
of monic Eisenstein polynomials of fixed degree $d$ and  of height
at most $H$, as $H\to \infty$. In particular, we give an explicit
bound for the error term. We also obtain an asymptotic formula
for arbitrary Eisenstein polynomials of height at most $H$.
\end{abstract}

\section{Introduction}
The Eisenstein criterion~\cite{Eis} is a simple well-known sufficient criterion to establish that an integer coefficient polynomial (and hence a polynomial with rational coefficients) is irreducible, see also~\cite{Cox}.
We recall that
\begin{equation}
\label{eq:poly}
f(X) = a_dX^d +a_{d-1}X^{d-1}+ \dots+a_1X+a_0 \in \Z[X]
\end{equation}
is called an {\it Eisenstein polynomial\/} if for some prime $p$ we have
\begin{enumerate}
\item[(i)] $p \mid a_i$ for $i=0, \ldots, d-1$,
\item[(ii)] $p^2\nmid a_0$,
\item[(iii)] $p \nmid a_d$.
\end{enumerate}

For integers $d \ge 2$ and $H \ge 1$, we let $\cE_d(H)$ be the set of all Eisenstein polynomials
with $a_d=1$ and of height at most $H$, that is, satisfying
$\max\{|a_{0}|,\ldots,|a_{d-1}|\}\leq H$.

Dubickas~\cite{Dub} 
has given  an asymptotic formula
for the cardinality $E_d(H) = \#\cE_d(H)$, see also~\cite{DoJo}.
Here we address this question again and obtain a more precise version
of this result with an explicit error term. Using techniques different to those in~\cite{Dub}, we also obtain an asymptotic formula for the number of polynomials, whether monic or non-monic, 
that satisfy the Eisenstein criterion.

\begin{thm}
\label{thm:main}
We have,
$$
E_d(H)=\vartheta_d 2^d H^d+\left\{\begin{array}{ll}
O\(H^{d-1}\),&\quad \text{if $d>2$}, \\
O(H(\log H)^2),&
\quad  \text{if $d=2$},
\end{array}\right.
$$
where
$$
\vartheta_d =  1-\prod_{p~\mathrm{prime}}
\(1- \frac{p-1}{p^{d+1}}\).
$$
\end{thm}

We remark that our argument is quite similar to that of Dubickas~\cite{Dub},
and in fact the method of~\cite{Dub} can also produce a bound on the error term
in an asymptotic formula for $E_d(H)$.
However we truncate the underlying inclusion-exclusion formula differently. 
This allows us to get a better bound on the error term than that
which follows from the approach of~\cite{Dub}.

Furthermore, we obtain an asymptotic formula for the cardinality $F_d(H) = \#\cF_d(H)$
of the set   $\cF_d(H)$ of Eisenstein polynomials
of the form~\eqref{eq:poly} of height at most $H$, that is, satisfying
$\max\{|a_{0}|,\ldots,|a_{d}|\}\leq H$. This result does not seem to have
any predecessors.

\begin{thm}
\label{thm:general}
We have,
$$
F_d(H)=\rho_d 2^{d+1} H^{d+1}+\left\{\begin{array}{ll}
O\(H^{d}\),&\quad \text{if $d>2$}, \\
O(H^2(\log H)^2),&
\quad  \text{if $d=2$},
\end{array}\right.
$$
where
$$
\rho_d =  1-\prod_{p~\mathrm{prime}}
\(1- \frac{(p-1)^2}{p^{d+2}}\).
$$
\end{thm}

\section{Notation}

As usual, for any integer $n\ge 1,$ let $\omega(n),\tau(n)$ and $\varphi(n)$ be the number of distinct prime factors, the number of divisors
and Euler function respectively (we also set $\omega(1) =0$).

We also use $\mu$ to denote the M{\" o}bius function, that is,
$$
\mu(n)= \begin{cases} (-1)^{\omega(n)} & \text{if } n \
\text{is square free}, \\
0 & \text{if } n \ \text{otherwise}.
\end{cases}
$$

Throughout the paper, the implied constants in the symbol `$O$' 
may occasionally,
where obvious, depend  on   the degree $d$.
We recall that the notation $U = O(V)$  is
equivalent to the assertion that the inequality $|U|\le c|V|$ holds for some
constant $c>0$. In addition to using $d$ to indicate the degree of a polynomial we retain the traditional use of $d,$ the divisor, as the index of summation in 
some well-known identities.

\section{Preparations}

We start by deriving a formula for the number of monic polynomials for which a given positive number satisfies conditions that are similar, but not equivalent, to the Eisenstein criterion.
Let $s$ be a positive integer. Let $\cG_d(s,H)$ be the set of monic polynomials~\eqref{eq:poly}
of height at most $H$ and
such that
\begin{enumerate}
\item[(i)] $s \mid a_i$ for $i=0, \ldots, d-1$,
\item[(ii)] $\gcd\(a_0/s,s\)=1$.
\end{enumerate}

It is easy to see that~\cite[Lemma~2]{Dub} immediately
implies the following result.

\begin{lem}
\label{lem:GdsH}
For $s \le H$, we have
$$
\# \cG_d(s,H)=\frac{2^dH^d\varphi(s)}{s^{d+1}}+
O\(\frac{H^{d-1} 2^{\omega(s)}}{s^{d-1}}\).
$$
\end{lem}

We  now derive a version of Lemma~\ref{lem:GdsH} for arbitrary polynomials.
 Let $\cH_d(s,H)$  be the set of polynomials~\eqref{eq:poly}
of height at most $H$ and
such that
\begin{enumerate}
\item[(i)] $s \mid a_i$ for $i=0, \ldots, d-1$,
\item[(ii)] $\gcd\(a_0/s,s\)=1$,
\item[(iii)] $\gcd(a_d,s)=1$.
\end{enumerate}

We also use the well-known  identity
\begin{equation}
\label{eq:muphi}
\sum_{d \mid s}\frac{\mu{(d)}}{d}=\frac{\varphi{(s)}}{s};
\end{equation}
see~\cite[Section~16.3]{HaWr}.

We now   define the following generalisation of the Euler function,
$$
\varphi(s,H) = \sum_{\substack{|a|\leq H \\ \gcd(a,s)=1}}1,
$$
and use the following well-known consequence of the sieve of
Eratosthenes.

\begin{lem}
\label{lem:Erat}
For any  integer $s \ge 1$, we have
$$
\varphi(s,H)
= \frac{2H\varphi (s) }{s} + O\(2^{\omega(s)}\).
$$
\end{lem}

\begin{proof} Using  the inclusion-exclusion principle
we write
$$
\varphi(s,H)
= \sum_{d \mid s}\mu (d)\sum_{\substack{|a|\leq H\\ d \mid a}} 1
=  \sum_{d \mid s}\mu (d)\(2\fl{\frac{H}{d}}+1\).
$$
Therefore,
\begin{eqnarray*}
\varphi(s,H)
= \sum_{d \mid s}\mu (d)\(\frac{2H}{d} + O(1)\)
= 2H \sum_{d \mid s} \frac{\mu (d)}{d}  + O\( \sum_{d \mid s}  |\mu(d)|\).
\end{eqnarray*}
Recalling~\eqref{eq:muphi} and that
$$
\sum_{d \mid s}  |\mu(d)| = 2^{\omega(s)},
$$
see~\cite[Theorem ~264]{HaWr},
we obtain the desired result.
\end{proof}

We also recall that
\begin{equation}
\label{eq:omega}
2^{\omega(s)} \le \tau(s) = s^{o(1)}
\end{equation}
as $s\to \infty$, see~\cite[Theorem~317]{HaWr}. 

Next we obtain an asymptotic formula for $\#\cH_d(s,H)$.

\begin{lem}
\label{lem:HdsH}
For $s \le H$, we have
$$
\#\cH_d(s,H)=\frac{2^{d+1}H^{d+1}\varphi^2(s)}{s^{d+2}}+
O\( \frac{H^{d }}{s^{d -1}} 2^{\omega(s)}\). $$
\end{lem}

\begin{proof}Fix a $d>1$.
For every $i=1, \ldots, d-1$, the number of admissible values of $a_{i}$
(that is, with $|a_i| \le H$ and $s \mid a_i$) is equal to
\begin{equation}
\label{eq:Coeff 1 d-1}
2\left\lfloor\frac{H}{s}\right\rfloor+1=\frac{2H}{s}+O(1).
\end{equation}

We now consider the admissible values of $a_0$ .
Writing $a_0 = sm$ with an integer $m$ satisfying $|m|\le H/s$ and $\gcd(m,s)=1$
we see from Lemma~\ref{lem:Erat}
that $a_0$   takes
\begin{equation}
\label{eq:Coeff 0}
\varphi(s, \fl{H/s}) =
\frac{2H\varphi{(s)}}{s^2}+O\(2^{\omega(s)}\)
\end{equation}
distinct values.

Lemma~\ref{lem:Erat} also implies that $a_d$   takes
\begin{equation}
\label{eq:Coeff d}
\varphi(s, H) =
\frac{2H\varphi{(s)}}{s}+O\(2^{\omega(s)}\)
\end{equation}
distinct values.

Combining~\eqref{eq:Coeff 1 d-1}, \eqref{eq:Coeff 0}
and~\eqref{eq:Coeff d} we obtain
\begin{equation}
\label{eq:prelim}
\begin{split}
\#\cH_d(s,H) & = \( \frac{2H}{s}
+O(1)\)^{d-1}
\(\frac{2H\varphi{(s)}}{s^2}+O\(2^{\omega(s)}\)\) \\
& \qquad\qquad\qquad\qquad\qquad \qquad\qquad
 \(\frac{2H\varphi(s)}{s}+O\(2^{\omega(s)}\)\) \\
 & = \( \(\frac{2H}{s}\)^{d-1}
+O\( \(\frac{H}{s}\)^{d-2}\) \)
\(\frac{2H\varphi{(s)}}{s^2}+O\(2^{\omega(s)}\)\) \\
& \qquad\qquad\qquad\qquad\qquad \qquad\qquad
 \(\frac{2H\varphi(s)}{s}+O\(2^{\omega(s)}\)\).
\end{split}
\end{equation}
Hence, using the trivial bound $\varphi(s) \le s$ and
that by~\eqref{eq:omega} we have $2^{\omega(s)} = O( H)$,
 we see that
\begin{equation*}
\begin{split}
\(\frac{2H\varphi{(s)}}{s^2}+O\(2^{\omega(s)}\)\)
 \(\frac{2H\varphi(s)}{s}+O\(2^{\omega(s)}\)\)
 = \frac{4H^2\varphi^2(s)}{s^3} + O\(H2^{\omega(s)}\).
\end{split}
\end{equation*}
Substituting into ~\eqref{eq:prelim}, and using that $\varphi(s) \le s$ again,
we obtain
$$
\#\cH_d(s,H)=\frac{2^{d+1}H^{d+1}\varphi^2(s)}{s^{d+2}}+  O\( \frac{H^{d}}{s^{d-1}} +  \frac{H^{d-1}}{s^{d-2}} 2^{\omega(s)}
+ \frac{H^{d }}{s^{d -1}} 2^{\omega(s)}\) .
$$
Taking into account that $s \le H$, we conclude  the proof.
\end{proof}

\section{Proof of Theorem~\ref{thm:main}}

We now prove the main result for monic Eisenstein polynomials.

The inclusion-exclusion principle implies that
$$
E_d(H) =-\sum_{s=2}^H\mu(s)\:\# \cG_d(s,H).
$$
Substituting the asymptotic formula of Lemma~\ref{lem:GdsH} for $\# \cG_d(s,H)$,  yields
\begin{equation}
\label{eq:Asymp}
\begin{split}
E_d(H)&=-\sum_{s=2}^H\mu(s)\left(\frac{2^dH^d\varphi(s)}{s^{d+1}}\right)+
O\(\sum_{s=2}^H \(\frac{H}{s}\)^{d-1}2^{\omega(s)}\)\\
&=-2^dH^d\sum_{s=2}^\infty\frac{\mu(s)\varphi(s)}{s^{d+1}}
+O\( H^d\sum_{s=H+1}^\infty\frac{\varphi(s)}{s^{d+1}} +H^{d-1}\sum_{s=2}^H
\frac{2^{\omega(s)}}{s^{d-1}}\)
\end{split}
\end{equation}
(since $\varphi(s)\le s$, the series in the main term  converges absolutely for $d \ge 2$).
Furthermore, since $\mu(s)\varphi(s)/s^{d+1}$ is a multiplicative function,
it follows that
\begin{equation}
\label{eq:Main}
\begin{split}
 -\sum_{s=2}^\infty\frac{\mu(s)\varphi(s)}{s^{d+1}} =
1 -&\sum_{s=1}^\infty\frac{\mu(s)\varphi(s)}{s^{d+1}}\\
= 1-\prod_{p~\mathrm{prime}}&
\left(1-\frac{\varphi(p)}{p^{d+1}}\right)  =  1-\prod_{p~\mathrm{prime}}
\left(1-\frac{p-1}{p^{d+1}}\right) .
\end{split}
\end{equation}
We also have
\begin{equation}
\label{eq:Err1}
\sum_{s=H+1}^\infty\frac{\varphi(s)}{s^{d+1}}
\le \sum_{s=H+1}^\infty\frac{1}{s^{d}} = O\(H^{-d+1}\).
\end{equation}

Recalling~\eqref{eq:omega},  for $d > 2$ we immediately obtain
\begin{equation}
\label{eq:Err2 1}
\begin{split}
\sum_{s=2}^H \frac{2^{\omega(s)}}{s^{d-1}} =  O(1).
\end{split}
\end{equation}

For $d=2$ we recall that
$$
\sum_{s \le t} 2^{\omega(s)} \le \sum_{s \le t}  \tau(s)
= (1 + o(1)) t \log t
$$
as $t \to \infty$, see~\cite[Theorem~320]{HaWr}.

Thus, via partial summation, we derive
\begin{equation}
\label{eq:Err2 2}
\begin{split}
\sum_{s=2}^H \frac{2^{\omega(s)}}{s} =  O \(\sum_{t=2}^H \frac{\log t}{t} \)
= O\((\log H)^2\).
\end{split}
\end{equation}
Substituting~\eqref{eq:Main}, \eqref{eq:Err1}, \eqref{eq:Err2 1} and~\eqref{eq:Err2 2}
in~\eqref{eq:Asymp}, we conclude the proof.

\section{Proof of Theorem~\ref{thm:general}}

The inclusion exclusion principle implies that
\begin{align*}
\#\cF_d(H)&= -\sum_{s=2}^H \mu(s)\#\cH_d(s,H).
\end{align*}
Using the asymptotic formula of  Lemma~\ref{lem:HdsH} yields
\begin{equation}
\label{eq:Asympg}
\begin{split}
\#\cF_d(H)&=-\sum_{s=2}^H \mu(s)\left(\frac{2^{d+1}H^{d+1}\varphi^2(s)}{s^{d+2}}\right)+O\left(\sum_{s=2}^H\frac{H^d\,2^{\omega(s)}}{s^{d-1}}\right)\\
&=-2^{d+1}H^{d+1}\sum_{s=2}^\infty\frac{\mu(s)\varphi^2(s)}{s^{d+2}}\\
&\qquad \qquad \qquad +O\left(H^{d+1}\sum_{s=H+1}^\infty \frac{\varphi^2(s)}{s^{d+2}}+H^d\sum_{s=2}^H \frac{2^{\omega(s)}}{s^{d-1}}\right)
\end{split}
\end{equation}
(since $\varphi(s) \leq s$, the series in the main term converges absolutely for $d \geq 2).$
In a similar manner to that used for \eqref{eq:Main}, we note that $\mu(s)\varphi^2(s)/s^{d+2}$ is a multiplicative function, so it follows that
\begin{equation}
\label{eq:general}
\begin{split}
 -\sum_{s=2}^\infty\frac{\mu(s)\varphi^2(s)}{s^{d+2}} =
1 -&\sum_{s=1}^\infty\frac{\mu(s)\varphi^2(s)}{s^{d+2}}\\
= 1-\prod_{p~\mathrm{prime}}&
\left(1-\frac{\varphi^2(p)}{p^{d+2}}\right)  =  1-\prod_{p~\mathrm{prime}}
\left(1-\frac{(p-1)^2}{p^{d+2}}\right) .
\end{split}
\end{equation}
Since $\varphi(s) \leq s,$ we also have
\begin{equation}
\label{eq:Err2g}
\sum_{s=H+1}^\infty\frac{\varphi^2(s)}{s^{d+2}}
\le \sum_{s=H+1}^\infty\frac{1}{s^{d}} = O\(H^{-d+1}\).
\end{equation}
Substituting~\eqref{eq:general}, \eqref{eq:Err2g},
in~\eqref{eq:Asympg}, and recalling~\eqref{eq:Err2 1} and~\eqref{eq:Err2 2},
we conclude the proof.
%
%
%
%
%
%
%

\section{Further Comments on $\vartheta_d$ and $\rho_d$}

Clearly, as $d\to \infty$,
\begin{align*}
 \vartheta_d&= 1-\prod_{p~\mathrm{prime}}\(1- \frac{p-1}{p^{d+1}}\)=\sum_{s=2}^\infty\frac{\mu(s)\varphi(s)}{s^{d+1}}\\
 &=\frac{1}{2^{d+1}} - \frac{2}{3^{d+1}} +\sum_{s=4}^\infty\frac{\mu(s)\varphi(s)}{s^{d}}=
 \frac{1}{2^{d+1}} - \frac{2}{3^{d+1}} + O\(\int_{3}^\infty\frac{1}{\sigma^{d-1}}d\sigma\)\\
&=\frac{1}{2^{d+1}} - \frac{2}{3^{d+1}} +  O\(\frac{1}{d3^{d}}\)
= \frac{1}{2^{d+1}}   +  O\(\frac{1}{3^{d}}\).
\end{align*}
Similarly,
$$
\rho_d = \frac{1}{2^{d+2}}   +  O\(\frac{1}{3^{d}}\), \qquad d\to \infty.
$$

We have computed  in Table~\ref{tab:Table} the approximate values of $\vartheta_d$ and $\rho_d$ for $d =2,\ldots, 10$.
The first 10,000 primes have been used in the calculations. The values of
$\vartheta_d$ are consistent with those given in~\cite{Dub}, but the table of the values of $\rho_d$ seems
to be new.

\begin{table}
\begin{center}
    \caption{Approximate values of $\vartheta_d$ and $\rho_d$ for $d =2,\ldots, 10$.}  \label{tab:Table}
   \medskip

    \begin{tabular}{ | l | l | l |}
    \hline
    \textrm{$d$} & $\vartheta_d$ & $\rho_d$  \\ \hline
    $2$ & 0.2515 &0.1677\\ \hline

    $3$ & 0.0953 & 0.0556 \\ \hline

    $4$&0.0409& 0.0224  \\ \hline

    $5$&0.0186& 0.0099 \\ \hline
   $6$& 0.0088 & 0.0046 \\ \hline
    $7$&0.0042 & 0.0022 \\ \hline
    $8$&0.0021&0.0010 \\ \hline
    $9$&0.0010 & 0.0005 \\ \hline
    $10$&0.0005&0.0003 \\ \hline
    \end{tabular}
  \end{center}
 \end{table}

\end{document}